\documentclass[11pt,amssymb,amsmath]{article}
\usepackage{epsf}
\usepackage{amssymb}
\usepackage{graphics}
\usepackage{amsmath}
\usepackage{amsthm}
\usepackage{amsfonts}
\usepackage{latexsym}
\usepackage{amssymb}
\usepackage{graphicx}
\usepackage{cyrillic}

\begin{document}

\newcommand{\cyrrm}{\fontencoding{OT2}\selectfont\textcyrup}
\newcommand{\cyrit}{\fontencoding{OT2}\selectfont\textcyrit}
\newcommand{\om}{\omega}

\newcommand{\beqn}{\begin{eqnarray}}
\newcommand{\eeqn}{\end{eqnarray}}
\newcommand{\fr}{\frac}
\newcommand{\nin}{\noindent}
\newcommand{\beq}{\begin{equation}}
\newcommand{\eeq}{\end{equation}}
\newtheorem{lemma}{lemma}[section]
\newtheorem{defn}[lemma]{Definition}
\newtheorem{thm}[lemma]{Theorem}
\newtheorem{prop}[lemma]{Proposition}
\newtheorem{cor}[lemma]{Corollary}
\newtheorem{remark}[lemma]{Remark}
\newtheorem{conjecture}[lemma]{Conjecture}
\newtheorem{example}[lemma]{Example}
\newtheorem{problem}[lemma]{Problem}
\newtheorem{defn-prop}[lemma]{Definition--Proposition}
\newtheorem{question}[lemma]{Question}

\newenvironment{rem}{\smallskip\noindent%
\refstepcounter{subsection}%
{\bf \thesubsection}~~{\sc Remark.}\hspace{-1mm}}
{\smallskip}

\newenvironment{ex}{\smallskip\noindent%
\refstepcounter{subsection}%
{\bf \thesubsection}~~{\sc Example.}}
{\smallskip}

\newenvironment{que}{\smallskip\noindent%
\refstepcounter{subsection}%
{\bf \thesubsection}~~{\sc Question.}\hspace{-1mm}}
{\smallskip}

\renewcommand{\theequation}{\arabic{equation}}

\newcommand{\Vir}{\text{\upshape Vir}}
\newcommand{\Rot}{\text{\upshape Rot}}
\newcommand{\Diff}{\text{\upshape Diff}}
\newcommand{\Vect}{\text{\upshape Vect}}
\newcommand{\Met}{\text{\upshape Met}}
\newcommand{\Metmu}{\Met_{\mu}}
\newcommand{\Vol}{\text{\upshape Dens}}
\newcommand{\SDiff}{Diff_\mu}
\newcommand{\SVect}{\Vect_\mu}
\newcommand{\Ham}{\text{\upshape Ham}}
\newcommand{\HamM}{\Ham_\omega(M)}
\newcommand{\ham}{\text{\upshape ham}}
\newcommand{\hamM}{\ham_\omega(M)}
\newcommand{\Symp}{\text{\upshape Symp}}
\newcommand{\symp}{\text{\upshape symp}}
\newcommand{\SDiffM}{\Diff_\mu(M)}
\newcommand{\SympM}{\Symp_\omega(M)}
\newcommand{\sympM}{\symp_\omega(M)}
\newcommand{\Diffmu}{\Diff_{\mu}}
\newcommand{\Isog}{\text{Iso}_g}
\newcommand{\DiffM}{\Diff(M)}
\newcommand{\MetM}{\Met(M)}
\newcommand{\MetmuM}{\Metmu(M)}
\newcommand{\VolM}{\Vol(M)}
\newcommand{\DiffmuM}{\Diffmu(M)}
\newcommand{\R}{\mathbb R}
\newcommand{\C}{\mathbb C}

\makeatletter
\newcommand\ackname{Acknowledgements}
\if@titlepage
  \newenvironment{acknowledgements}{%
      \titlepage
      \null\vfil
      \@beginparpenalty\@lowpenalty
      \begin{center}%
        \bfseries \ackname
        \@endparpenalty\@M
      \end{center}}%
     {\par\vfil\null\endtitlepage}
\else
  \newenvironment{acknowledgements}{%
      \if@twocolumn
        \section*{\abstractname}%
      \else
        \small
        \begin{center}%
          {\bfseries \ackname\vspace{-.5em}\vspace{\z@}}%
        \end{center}%
        \quotation
      \fi}
      {\if@twocolumn\else\endquotation\fi}
\fi
\makeatother


\newcounter{bk}
\newcommand{\bk}[1]
{\stepcounter{bk}$^{\bf BK\thebk}$%
\footnotetext{\hspace{-3.7mm}$^{\blacksquare\!\blacksquare}$
{\bf BK\thebk:~}#1}}

\newcounter{gm}
\newcommand{\gm}[1]
{\stepcounter{gm}$^{\bf GM\thest}$%
\footnotetext{\hspace{-3.7mm}$^{\blacksquare\!\blacksquare}$
{\bf GM\thest:~}#1}}


\title {Symplectic structures and dynamics\\on  vortex membranes}

\author{Boris Khesin\thanks{
Department of Mathematics,
University of Toronto, Toronto, ON M5S 2E4, Canada;
e-mail: \tt{khesin@math.toronto.edu}
}
}
\date{January 27, 2012}
\maketitle
\rightline{\it To the memory of Vladimir Igorevich Arnold}
\bigskip

\begin{abstract}
We present a Hamiltonian framework for higher-dimensional vortex filaments (or membranes) and vortex sheets as singular 2-forms with support of codimensions 2 and 1, respectively, i.e. singular elements of the dual to the Lie algebra of divergence-free vector fields. It turns out that the localized induction approximation (LIA) of the hydrodynamical Euler equation describes the skew-mean-curvature flow on vortex membranes of codimension 2 in any $\R^n$, which generalizes to any dimension the classical binormal, or vortex filament, equation in $\R^3$.

This framework also allows one to define the symplectic structures on the spaces of vortex sheets, which interpolate between the corresponding structures on vortex filaments and smooth vorticities.
\end{abstract}

\tableofcontents

\bigskip

\section*{Preface}
Vladimir Arnold's 1966 seminal paper \cite{Arn66} in which he introduced 
numerous geometric ideas into hydrodynamics  influenced the field far beyond its original scope. One of Arnold's  remarkable and, in my opinion, very unexpected insights was to regard the fluid vorticity field (or the vorticity 2-form) as an element of the dual to
the Lie algebra of the fluid velocities, i.e., the algebra of divergence-free vector fields on the flow domain. 

In this paper, after a review of the concept of isovorticed fields, which was crucial, e.g., in Arnold's stability criterion in fluid dynamics, we present an ``avatar" of this concept,  providing a natural framework for the formalism of vortex membranes and vortex sheets. In particular, we present the equation of localized induction approximation, which turns out to be the skew-mean-curvature flow in any dimension. We also show that the space of vortex sheets has a natural symplectic structure and  occupies an intermediate position between vortex filaments in 3D (or point vortices in 2D) equipped with the Marsden-Weinstein symplectic structure on the one hand and smooth vorticity fields with the Lie-Poisson structure on them on the other hand.

\medskip

Before launching into hydrodynamical formalism in this memorial paper, I  would like to recall an episode with Vladimir Igorevich, related not to fluid dynamics, but rather to his equally surprising insights in real life: his remarks were always witty, to the point, and often mischievous.\footnote{At times it was hard to tell whether he was being serious or joking. For instance, when inviting the seminar participants for an annual ski trip outside of Moscow, Arnold would say:
``This time we are not planning too much, only about 60km. Those who doubt they could make it --- need not worry: the trail is so conveniently designed that one can return from interim bus stops on the way, which we'll be passing by every 20km."}

Back in 1986 Arnold became a corresponding member of the Soviet Academy of Sciences. This was the time of 
``glasnost" and ``acceleration": novels of many formerly forbidden authors appeared in print for the first time.
Jacques Chirac, Prime Minister of France at the time, visited Moscow and gave a speech in front of the Soviet Academy in the Spring of 1986. The speech was typeset beforehand and distributed to the Academy members.
Arnold was meeting us, a group of his students, right after Chirac's speech and brought us that printout.
Chirac, who knows Russian, mentioned almost every disgraced poet or writer of the Soviet Russia in his speech: it contained citations from Gumilev, Akhmatova, Mandelshtam, Pasternak... And on the top of this printout, above the speech, was the following epigraph in Arnold's unmistakable handwriting: 

\medskip

	              \leftline{\cyritshape  ``... Ya dopuwu: uspehi nashi bystry } 
                   \leftline{\cyritshape  No gde zh u nas ministr-demagog?}
		      \leftline{\cyritshape  Pustp1 proberut vse spiski i registry,}
		      \leftline{\cyritshape Ya pyatp1 ruble\u{i} bumazhnyh dam v zalog;}
                     \leftline{\cyritshape Bytp1 mozhet, ih vo Francii nemalo,}
                     \leftline{\cyritshape  No na Rusi ih net -- i ne byvalo!''}
\medskip

~\qquad {\cyritshape   A.K.~Tolsto\u{i} ``Son Popova" (1873)}\footnote{English translation by A.B.~Givental:\\
~\\   
``... Our nation's rise is, I concur, gigantic,\\
But demagogs among our statesmen?! Let\\
all rosters, in the fashion most pedantic,\\
be searched, I'd put five rubles for a bet:\\
There could be more than few in France or Prussia,\\
but are - and have been - none in mother Russia!"

{\it ~\quad from ``Popov's Dream" by A.K.~Tolstoy (1873)}}

\smallskip

Who could take this speech seriously after such a tongue-in-cheek epigraph?
As a curious aftermath, Arnold and Chirac shared the Russia State Prize in 2007. 
 
\bigskip

Returning to mathematics, I think that 
the ideas introduced by Arnold in \cite{Arn66}, so natural in retrospect, 
are in fact most surprising given the state of the art in hydrodynamics of the mid-60s 
both for their deep insight into the nature of fluids 
and their geometric elegance and simplicity. 
In the next section we begin with a brief survey of the use of vorticity in a few 
hydrodynamical applications and discuss how it helps in understanding 
the properties of vortex filaments and vortex sheets. 

\bigskip


\section{The vorticity form of the Euler equation}

Consider the Euler equation for an inviscid incompressible fluid 
filling a Riemannian manifold $M$ (possibly with boundary).  The fluid motion is described 
as an evolution of the fluid velocity field $v$  in $M$ which is governed by  the 
classical Euler equation:
\begin{equation}\label{ideal}
\partial_t v+(v, \nabla) v=-\nabla p\,.
\end{equation}
Here the field $v$ is assumed to be divergence-free (${\rm div}\, v=0$)  
with respect to the Riemannian volume form $\mu$ and tangent to the boundary of $M$. 
The pressure function $p$ is defined uniquely modulo an additive constant 
by these restrictions on the velocity $v$. The term
$(v, \nabla) v$ stands for the Riemannian covariant derivative 
$\nabla_v v$ of the field $v$ in the direction of itself.

\bigskip

\subsection{The Euler equation on vorticity}

The vorticity (or Helmholtz) form of the Euler equation is
\begin{equation}\label{idealvorticity}
\partial_t \xi+L_v \xi=0\,,
\end{equation}
where $L_v$ is the Lie derivative along the field $v$ and
which means that the vorticity field $\xi:={\rm curl}~v$ is transported by (or ``frozen into") the fluid flow.
In 3D the vorticity field $\xi$ can be thought of as a vector field, while in 2D it is a scalar vorticity function.
In the standard 2D-space with coordinates $(x_1,x_2)$ the vorticity function is curl~$v:={\partial v_2}/{\partial x_1}-{\partial v_1}/{\partial x_2}$, 
which can be viewed as the vertical coordinate of the vorticity vector field for the 2D plane-parallel flow in 3D.
The fact that the vorticity is ``frozen into" the flow allows one to define 
various invariants of the  hydrodynamical Euler equation,  e.g., the conservation of 
helicity in 3D and  enstrophies in 2D. 

The Euler equation has the following Hamiltonian formulation.  For an $n$-dimensional Riemannian manifold $M$ with a volume form $\mu$ consider the  Lie group $G=\SDiffM$ of volume-preserving 
diffeomorphisms of $M$. 
The corresponding Lie algebra $\mathfrak g={\SVect}(M)$ 
consists of smooth divergence-free vector fields in $M$ tangent to the boundary $\partial M$:  
$$
{\SVect}(M)=\{V\in \Vect(M)~|~L_V\mu=0\,{\rm and }\, V||\partial M\}\,.\footnote{We usually denote generic elements of ${\SVect}(M)$, as well as variation fields, by capital letters, while keeping the small $v$ notation for  velocity fields related to the dynamics.}
$$  
The natural ``regular dual" space for this Lie algebra is the space of cosets of smooth 1-forms on $M$ modulo exact 1-forms, $\mathfrak g^*=\Omega^1(M)/d\Omega^0(M)$, see e.g. \cite{AK, MW}. 
The pairing between cosets $[\eta]$ of 1-forms $\eta$ and vector fields $W\in {\SVect}(M)$ is given by 
\begin{equation}\label{pairing}
\langle [\eta],W\rangle:=\int_M i_W \eta\cdot \mu\,,
\end{equation}
where $i_W$ is the contraction of a differential form with a vector field $W$.
The Euler equation (\ref{ideal}) on the dual space assumes the form
$$
\partial_t[\eta]+L_v[\eta]=0\,,
$$
where $[\eta]\in \Omega^1(M)/d\Omega^0(M)$ stands for the coset of the 1-form $\eta=v^\flat$ related to the velocity vector field $v$ by means of the Riemannian metric on $M$. (For a manifold $M$ equipped with a Riemannian metric $(.,.)$ one defines the 1-form $v^\flat$ as the pointwise inner product with 
vectors of the velocity field $v$:
$v^\flat(W): = (v,W)$ for all $W\in T_xM$, see details in \cite{Arn66, AK}.) 

Instead of dealing with cosets of 1-forms, it is often more convenient to pass to their differentials.
The {\it vorticity 2-form} $\xi:=dv^\flat$ is the differential of the 1-form $\eta=v^\flat$.
Note that  in 3D the vorticity vector field $\mathrm{curl}~v$ is defined by the 2-form $\xi$ via $i_{\mathrm{curl}\,v} \mu=\xi$ for the volume form $\mu$. 
In 2D $\mathrm{curl}~v$ is the function $\mathrm{curl}~v:=\xi/\mu$.
The definition of vorticity $\xi$ as an exact 2-form in $M$ makes sense for any dimension of the manifold $M$. This point of view can be traced back to the original papers 
by Arnold, see e.g. \cite{ArnHamiltonian, arn}

Such a definition immediately implies that: 

$i)$ the vorticity 2-form $\xi:=d\eta$ is well-defined for cosets $[\eta]$: 1-forms $\eta$ in the same coset have equal vorticities, and 

$ii)$ the Euler equation in the form (\ref{idealvorticity})
or $\partial_t (d\eta)+L_v (d\eta)=0$ means that the vorticity 2-form $\xi=d\eta$ is transported by (or frozen into) the fluid  flow in {\it any} dimension. The latter allows one to define  generalized enstrophies 
for all even-dimensional flows  and  helicity-type integrals for all
odd-dimensional ideal fluid flows, which turn out to be first integrals of the corresponding 
higher-dimensional Euler equation, see e.g. \cite{AK}.
This geometric setting can be rigorously developed within the Sobolev framework for $H^s$ diffeomorphisms and vector fields on $M$ for sufficiently large $s$, see \cite{EM}. To present the geometric ideas  and  include singular vorticities  we  keep things formal in what follows.
\bigskip

\begin{remark} 
{\rm
This point of view on vorticity was the basis for Arnold's stability criterion. Namely, steady fluid flows are critical points of the
  restriction of the Hamiltonian (which is the kinetic energy function defined on the dual space) to the spaces of isovorticed fields, i.e. sets of fields with  diffeomorphic vorticities. If the restriction of the Hamiltonian functional has a sign-definite (positive or negative) second variation at the critical point, the corresponding  steady flow is Lyapunov stable.
This is  famous Arnold's stability test. In particular, he proved (see e.g. \cite{arn, AK}) that shear flows in an annulus with no inflection points in the velocity profile are Lyapunov stable, thus generalizing the Rayleigh stability condition.
}
\end{remark}

\bigskip

\subsection{Smooth and singular vorticities}

Another consequence of such a point of view on vorticity, which is of the main interest to us, is the existence of the Poisson structure. 

Let $M$ be an $n$-dimensional Riemannian manifold with a volume form $\mu$
and filled with an incompressible fluid.  As we discussed above,
the vorticity of a fluid motion geometrically is  the 2-form
defined by $\xi:= dv^\flat $, where $v^\flat$ is the 1-form obtained from the vector field $v$ by the metric lifting of indices.
Assume that $H^1(M)=0$ to simplify the reasoning below. 
Then the space of vorticities $\{\xi\}$, i.e. the space of exact 2-forms $d\Omega^1(M)$, 
coincides with the
dual space to the Lie algebra $\Vect_\mu(M)$ of divergence-free vector
fields. Indeed, $\Vect_\mu(M)^*\simeq \Omega^1/d\Omega^0\simeq d\Omega^1$ where the latter identification holds since $H^1(M)=0$.

\smallskip

\begin{remark}
{\rm
As the dual space to a Lie algebra,  the  space of vorticities $\Vect_\mu(M)^*=\{\xi\}$ has the natural Lie-Poisson structure.
Its symplectic leaves are coadjoint orbits of the corresponding group $\SDiffM$.
Here such orbits are sets of fields with  diffeomorphic vorticities on $M$,
with the group action being the action of volume-preserving diffeomorphisms on  vorticity 2-forms. The Euler equation defines a Hamiltonian evolution on these orbits.

 The corresponding (Kirillov-Kostant) symplectic structure on orbits in $\SVect(M)^*$  
 is given by the following formula. 
 Let $V$ and $W$ be two divergence-free vector fields in $M$, which we regard as a pair 
of variations of the point $\xi$ in $\SVect(M)^*$. The {\it Kirillov-Kostant symplectic structure} on  coadjoint orbits  associates to a pair of such variations
tangent to the coadjoint orbit of the vorticity $\xi$ the following quantity:
\begin{equation}\label{KK}
\omega^{KK}_\xi(V,W):=\langle d^{-1}\xi,~~[V,W] \rangle 
=\langle \eta, ~[V,W] \rangle =\int_{M} \eta \wedge i_{[V,W]} \mu=\int_M \xi \wedge i_Vi_W\mu\, .
\end{equation}
Here the 1-form $\eta=d^{-1}\xi$ is a primitive of the vorticity 2-form $\xi$, and
$[V,W]$ is the commutator of the vector fields $V$ and $W$  in $M$. Note that for divergence-free vector fields $V$ and $W$ their commutator satisfies the identity $i_{[V,W]}\mu=d(i_Vi_W\mu)$, which implies the last equality in (\ref{KK}).
}
\end{remark}

\medskip
\begin{figure*}
\begin{center}
\begin{tabular}{c|c|c|c|c}
support & vorticity & symplectic & evolution & Hamiltonian\\
codim & types & structure & equation & ~\\
\hline
\hline
   &   &  & & \\
   &    smooth          &   $\omega^{KK}_\xi(V,W)$& vorticity & energy\\
0 &  vorticities $\xi$ & $=\int_M \xi\wedge i_V i_W\mu $ & Euler equation&  $H=\frac 12\int_M (v,v)\,\mu$ \\
   &  &  & $\partial_t \xi=-L_v \xi$ & \\
\hline
   &   &  & & \\
   &  vortex  & $\omega_{\partial_\Gamma\wedge\alpha}(V,W)$&Euler $\Rightarrow$ Birkhoff-Rott & \\
1 & sheets  $\partial_\Gamma\wedge\alpha$ &  $=\int_\Gamma \alpha\wedge  i_Vi_W \mu$ & LIA -- ?& $H=\,?$\\
   &   &  & & \\
\hline
   &   &  & & \\
    &   2D: point vortices    &   $\omega_{(\kappa_j,z_j)} $ & Euler $\Rightarrow$
    Kirchhoff & $H=$ Kirchhoff   \\
 &   $\sum \kappa_j \delta_{z_j}$   &  $=\sum \kappa_j \,dx_j\wedge dy_j$ & LIA=0 & 
Hamiltonian $\mathcal H$ \\
    &   ---------------------   &   ---------------------  &  --------------------- & 
  ---------------------  \\
2 & 3D:  filaments &    $\omega^{MW}_\gamma (V,W)$  & LIA: binormal eqn & \\  
    &  $C\cdot \delta_\gamma$  &  $=\int_\gamma i_Vi_W\mu$ & 
    $\partial_t\gamma=\gamma'\times\gamma''$&  $H={\rm length}(\gamma)$\\
    &   ---------------------   &   ---------------------  &  --------------------- &   
  ---------------------  \\
& any D:  membranes  & $\omega^{MW}_P (V,W)$   & LIA: skew mean  & \\
& (higher filaments) &    $=\int_P i_Vi_W\mu$   & curvature flow &$H={\rm volume}(P)$\\
& $C\cdot \delta_P $ &      &  $\partial_t P=J({\bf MC}(P))$  &
\end{tabular}
\end{center}
\end{figure*}

In this paper we deal with singular vorticities.
Regular vorticities have support of full dimension, i.e. of codimension 0 in $M$, while singular ones have
support of codim~$\geq 1$.
Singular vorticities form a subspace in (a completion of) the  dual space 
$\Vect_\mu(M)^*=d\Omega^1(M)$.
Note that since vorticity is a (possibly singular) 2-form (more precisely, a 
  current of degree 2), its support has to be
of codim~$\leq 2$. (E.g., if support is of codim~$=3$, it corresponds to a
  singular 3-form. We refer to \cite{Federer} for details on  currents.)
  
The most interesting cases of support are of codimension 1 (vortex sheets) and
  codimension 2 (point vortices in 2D, vortex filaments in 3D and vortex membranes, or higher  filaments, for any dimension).
  In the next sections we  start with the codimension 2 case, and deal 
with the codimension 1 case towards the end of the paper. The main types of singular vorticities, as well as related to them symplectic structures and Hamiltonian equations studied below, are summarized in the  table above.
While the goal of this paper is partially expository, and various facts on vortex filament dynamics are scattered in the extensive literature,  certain results presented below (in particular, the Hamiltonian framework for vortex sheets, the skew-mean-curvature flows and the LIA in any dimension) are apparently new.

\bigskip


\smallskip

\section{Singular vorticities in codimension 2: point vortices and filaments}\label{filaments}

\bigskip


\subsection{Point vortices in 2D}

\medskip

Let $M$ be the 2-dimensional Euclidean plane $\R^2$.
Let the 2D vorticity $\xi$ be supported on $N$ point vortices:
$\xi=\sum^N_{j=1} \kappa_j\,\delta_{z_j}=\sum^N_{j=1} \kappa_j\,\delta( z- z_j)$, where
$ z_j=( x_j,  y_j)$ are coordinates of the $j$th point vortex 
in  $\R^2=\C^1$ with the standard area form  
$\mu=d  x\wedge d  y$. Kirchhoff's theorem states that the evolution 
of vortices according to the Euler equation is described by the system
\begin{equation}\label{kirch}
\kappa_j \dot x_{j}
=\frac{\partial \mathcal H}{\partial y_{j}},\qquad
\kappa_j \dot y_{j}
=-\frac{\partial \mathcal H}{\partial x_{j}},\qquad 
1\le j\le N\,.
\end{equation}
This is a Hamiltonian system in  $\R^{2N}$ with the  Hamiltonian function
$$
\mathcal H=-\frac{1}{4\pi}\sum^N_{j<k}\kappa_j\kappa_k \, 
\ln | z_j-  z_k|^{2}
$$
and the Poisson structure is given by the bracket
\begin{equation}\label{poissonvortices}
\{f,g\}=\sum^N_{j=1}\frac{1}{\kappa_j}
\left( \frac{\partial f}{\partial  x_j}
\frac{\partial g}{\partial  y_j}
-\frac{\partial f}{\partial  y_j}
\frac{\partial g}{\partial  x_j}\right)\,.
\end{equation}
\medskip

One can derive the above Hamiltonian dynamics from the  2D Euler equation  in the vorticity form
$
\partial_t \xi =\{\psi,\xi\}\,,
$
where $\xi$ is a vorticity function in $\R^2$ and
the stream function (or Hamiltonian) $\psi$ of the flow satisfies
$\Delta \psi=\xi$, see e.g. \cite{Khesin}. While the system (\ref{kirch}) goes back to Kirchhoff,  its properties 
for various numbers of point vortices and versions for different manifolds have been of constant interest, see e.g. \cite{Kimura, Khesin}. The cases of $N=2$ and $N=3$ point vortices are integrable, while those of $N\ge 4$ are not. The subtle issue of in what sense the equation of $N$ point vortices approximates the 2D Euler equation as $N\to\infty$ is  treated, e.g., in \cite{MP}.\footnote{The Hamiltonian system for the point vortex approximation of the 2D Euler equation is reminiscent of the Calogero-Moser system for  the  evolution of poles of rational solutions of the KdV equation.}

The origin of the Poisson bracket (\ref{poissonvortices}) is explained by the following

\begin{prop}{\rm (\cite{MW})}
The Poisson bracket  (\ref{poissonvortices}) is defined by the Kirillov-Kostant symplectic structure 
on the coadjoint orbit of the (singular) vorticity $\xi=\sum^N_{j=1} \kappa_j\,\delta( z- z_j)$
in  (the completion of) the dual of the Lie algebra  $\mathfrak g={\SVect}(\R^2)$ 
of divergence free vector fields in $\R^2$.
\end{prop}

\bigskip
\begin{proof} Indeed, a vector tangent to the coadjoint orbit of such a singular vorticity 
$\nu$ can be regarded as a 
collection of vectors $\{ V_j \}$ in $\R^2$  attached at points $z_j$.  Then for a pair of tangent vectors the corresponding Kirillov-Kostant symplectic structure becomes the weighted sum of the  corresponding contributions at each point vortex $z_j$ with strengths $\kappa_j$ as the  corresponding weights:
$$
\omega^{KK}_\xi(V,W):=\int_{\R^2} \xi \wedge i_Vi_W\mu
=\sum_j  \kappa_j\,\mu(V_j, W_j)\, .
$$
The Poisson bracket  (\ref{poissonvortices}), being the inverse of the symplectic structure, has the reciprocals of the weights $\kappa_j$. 
\end{proof}

\bigskip

\bigskip


\subsection{Vortex filaments in 3D}

By passing from 2D to 3D we move from point vortices to filaments. Vortex filaments are  curves in $\R^3$ being supports of singular vorticity fields. They are governed by the Euler equation
  \begin{equation}\label{xiequation}
  \partial_t \xi+L_v \xi=0\,,
  \end{equation}
  where $v={\rm curl}^{-1}\xi$ and the vorticity field (or a 2-form) $\xi$ has support on a curve $\gamma\subset \R^3$.
  (Note that the exactness of the form $\xi$ implies that $\gamma$ is a boundary of a 2-dimensional domain, i.e., in particular, its components are either closed  or go to infinity.)
The  Euler dynamics is nonlocal in terms of the vorticity field, or 2-form, $\xi$
  since it requires finding the field $v={\rm curl}^{-1}\xi$.
 
The  localized induction approximation (LIA)  of the vorticity motion is a procedure which 
allows one to keep only the local terms in the vorticity Euler equation, as we discuss below. In $\R^3$ the corresponding evolution is described by the {\it vortex filament equation} 
\begin{equation}\label{binormal}
\partial_t \gamma=\gamma'\times \gamma''\,,
\end{equation}
where 
$\gamma(\cdot, t)\subset \R^3$ is a time-dependent arc-length parametrized space curve. For an arbitrary parametrization the filament equation becomes $\partial_t\gamma=k \cdot \mathbf b$,
where $k $ and $\mathbf{ b=t\times n}$ stand, respectively, for
the  curvature value and binormal unit vector of the  curve $\gamma$ at the corresponding point.
This equation is often called the {\it binormal equation}.
(The equivalence of the two equation forms is straightforward: for an arc-length parametrization the tangent vectors
${\bf t}= \partial \gamma/\partial \theta=\gamma'$ have unit length and the acceleration vectors are
$\gamma''=\partial {\bf t}/\partial \theta=k\cdot {\bf n}$,
i.e. $\partial_t \gamma=\gamma'\times \gamma''$ becomes $\partial_t\gamma=k \cdot \mathbf b$,
where the latter equation is valid for an arbitrary parametrization.)

\begin{remark}
{\rm 
Here we briefly recall  the LIA derivation of the equation (\ref{binormal}) in 3D, see e.g. \cite{Calini}, and give more details in Section \ref{LIAanyD} and Appendix.

Assume that the velocity distribution $v$ in  $\R^3$
has vorticity $\xi=\mathrm{curl}\,v$  concentrated on a smooth embedded arc-length parametrized
  curve $\gamma\subset \R^3$ of length $L$. Then
$$
\xi(q,t)=C\int_0^L\delta(q-\gamma(\theta,t))\frac{\partial \gamma}{\partial \theta}\,d\theta\,.
$$
Here $\delta$ is the delta-function in $\R^3$ and the constant $C$, the strength of the filament, is the flux of $\xi$ across (or, which is the same,  the circulation of $v$ over) a small contour around the core of the vortex  filament $\gamma$. Note that the exactness of the 2-form $\xi$ also implies that the filament strength $C$ is indeed constant along $\gamma$.

The Biot-Savart law allows one to represent the velocity field
in terms of its vorticity:
\begin{equation}\label{3DBiot}
v(q,t)=-\frac{1}{4\pi}\int_M\frac{(q-\tilde q)\times \xi(\tilde q)}{\|q-\tilde q\|^3}\,d^3\tilde x
=-\frac{C}{4\pi}\int_\gamma\frac{q-\gamma(\tilde \theta, t)}{\|q-\gamma(\tilde \theta, t)\|^3}\times \frac{\partial \gamma}{\partial \tilde \theta}\,d\tilde \theta\,.
\end{equation}
By utilizing the fact that the time evolution of the curve $\gamma$ is given by the velocity field $v$ itself: 
$\frac{\partial \gamma}{\partial t}(\theta, t)=v(\gamma(\theta, t),t)$ we come to the 
following integral:
$$
\frac{\partial \gamma}{\partial t}(\theta, t)
=-\frac{C}{4\pi}\int_\gamma\frac{\gamma(\theta, t)-\gamma(\tilde \theta, t)}{\|\gamma(\theta, t)-\gamma(\tilde \theta, t)\|^3}\times \frac{\partial \gamma}{\partial \tilde \theta}\,d\tilde \theta\,.
$$
This integral is divergent with the main singularity coming from the points on the curve $\gamma$ close to each other on the curve (i.e. with small $|\tilde \theta-\theta |$). 
Given $\theta$ the Taylor expansion of $\tilde\gamma(\theta)$ in $(\tilde \theta-\theta)$ yields
$$
\frac{\partial \gamma}{\partial t}(\theta, t)
=\frac{C}{8\pi}\left[\frac{\partial \gamma}{\partial \theta}\times \frac{\partial^2 \gamma}{\partial \theta^2}\int_0^L\frac{d\tilde \theta}{|\theta-\tilde \theta|}+\mathcal O(1)\right]\,.
$$
Since the right-hand side is divergent, for a small $\epsilon$ we take the truncation of the
integral by considering only the part $|\tilde \theta-\theta |>\epsilon$ of the integration domain $[0,L]$.
The corresponding integral is of order $\ln \epsilon$. Rescaling time by means of 
$t\to t(C/8\pi)\ln\epsilon$   we are keeping only the singularity term and neglecting others as $\epsilon\to 0$.

This way we obtain the vortex filament equation (\ref{binormal}). It is also called the localized induction approximation (LIA) since the velocity field ${\partial \gamma}/{\partial t}$ of the curve $\gamma$ is induced by its own vorticity, i.e. vorticity supported on the curve, while only  parts of the curve sufficiently close to a given point $\gamma(\theta)$ 
determine the velocity field  at that point.  (We discuss  the above limit in higher dimensions in the next two sections.) Note that in 2D point vortices interact with each other but not with themselves (as manifested by the Kirchhoff Hamiltonian), i.e. the localization in 2D would give the zero LIA equation.
}
\end{remark}

\begin{remark}
{\rm 
This binormal equation is known to be Hamiltonian relative to the Marsden-Weinstein symplectic structure 
on non-parametrized space curves in $\R^3$.
Recall that the  {\it Marsden-Weinstein symplectic structure}  is defined {\it on oriented curves} $\gamma$ by
\begin{equation}\label{MWss}
\omega^{MW}_\gamma(V,W):=\int_\gamma i_V i_W\mu=\int_\gamma \mu (V,W,\gamma')\,d\theta
\end{equation}
where $V$ and $W$ are two vector fields attached to the curve $\gamma$ 
and regarded as variations of this curve, while the volume form $\mu$ is evaluated on the three vectors $V,W$ and $\gamma'={\partial \gamma}/{\partial \theta}$. One can see that 
this integral does not depend on the parameter $\theta$ on the curve $\gamma(\theta)$.

Equivalently, this symplectic structure can be defined by means of the operator $J$ of almost complex structure on curves: any variation, i.e. vector field
attached at the oriented curve $\gamma$, is rotated by the operator $J$ in the planes orthogonal to 
$\gamma$ by $\pi/2$ in the positive direction (which makes a skew-gradient from a gradient field), see details in \cite{MW, AK}.

One can show that this is the Kirillov-Kostant symplectic structure on the coadjoint orbit of the vorticity $\xi_\gamma$ supported on the curve $\gamma$ and understood as a point in a completion of the dual of the Lie algebra: $\xi_\gamma\in d\Omega^1(\R^3)=\SVect(\R^3)^*$.
The pairing of $\gamma$ and a divergence-free vector field $V$ can be defined directly as
$\langle \gamma, V\rangle:=\mathrm{Flux }\, V|_{\sigma}$, where $\sigma$ is an oriented surface whose boundary is $\gamma=\partial \sigma$.
}
\end{remark}

\begin{remark}\label{lengthH}
{\rm 
As discussed above, the Euler equation (\ref{xiequation}) is Hamiltonian  with the Hamiltonian  function given by the kinetic energy. The energy $E(v)=\frac 12\int_M (v,v)\,\mu$ is local in terms of velocity fields, but it is nonlocal in terms of vorticities: $E(\xi)=\frac 12 \int_M ( \mathrm{curl}^{-1}\xi,\mathrm{curl}^{-1}\xi)\,\mu$. 
It turns out that after taking the localized induction approximation, when we keep only the local terms, the filament equation remains Hamiltonian with respect to the same Marsden-Weinstein symplectic structure, but with a different Hamiltonian (see Sections \ref{MCanyD}, \ref{LIAanyD} and Appendix).
  
The corresponding new Hamiltonian functional turns out to be the length functional of the curve: 
  $H(\gamma)=\text{length}(\gamma)=\int_\gamma \|\gamma'(\theta)\|\,d\theta$, see e.g. \cite{AK}. 
  Indeed, the variational derivative, i.e. the ``gradient,'' of this length functional $H$ 
  is $\delta H/\delta\gamma =  -\gamma''=-\mathbf{t}' = - k \cdot\mathbf{n}$, where $\mathbf{t}$ and $\mathbf{n}$
  are, respectively, the unit tangent and normal fields to the curve $\gamma$. The dynamics is given by the corresponding skew-gradient,
  which is obtained from $\delta H/\delta\gamma$ by applying $J$ for the above symplectic structure.
This operator, rotating  the plane orthogonal to  $\mathbf{t}$ by $\pi/2$,
sends $-k \cdot\mathbf{n}$ to  $k \cdot\mathbf{b}$. 
In the next section we discuss how the Marsden-Weinstein symplectic structure 
and binormal equation are generalized to higher dimensions.

The LIA evolution is close to the actual Euler evolution of a vortex filament only for a short time (when the local term in dominant). For large times the LIA filament may, e.g., self-intersect, while the incompressible Euler dynamics has a frozen-in vorticity and it does not allow topology changes of the filaments.
}
\end{remark}

\bigskip


\section{Vortex membranes and skew-mean-curvature flow}\label{MCanyD}

For a smooth hypersurface in the Euclidean space $\R^n$ its mean curvature is a function 
on this surface. Similarly, one can define a mean curvature vector field 
for a smooth submanifold of any dimension $l$.

\begin{defn}\label{defMC}
{\rm a) Let $P$ be a smooth  submanifold of dimension $l$ in the Euclidean space $\R^n$.
Its second fundamental form at a point $p\in P$ is a map from the tangent space $T_pP$ to  the normal space $N_p P$. The  {\it mean curvature vector} ${\bf MC}(p)\in N_pP$ is the normalized trace of the second fundamental form at $p$, i.e. the trace divided by $l$.

\smallskip

b) Equivalently, the {\it mean curvature vector} ${\bf MC}(p)\in N_pP$ is the mean value of the curvature vectors of geodesics in $P$ passing through the point $p$ when we average over the sphere $S^{l-1}$ of all possible unit tangent vectors in $T_p P$ for these geodesics.
}
\end{defn}


\medskip

Now consider a closed oriented embedded submanifold (membrane)  $P$ of codimension 2 in $\R^n$ 
(or more generally, in a Riemannian manifold $M^n$) with $n\ge 3$
and the Marsden-Weinstein (MW) symplectic structure  on such submanifolds. 
Recall that the Marsden-Weinstein symplectic structure $\omega^{MW}$ on membranes of codimension 2 in $\R^n$ (or in any $n$-dimensional manifold) with a volume form $\mu$ is defined 
similar to the 3-dimensional case: two variations of a membrane $P$ are regarded as a pair of normal vector fields attached to the membrane $P$ and the value of the symplectic structure on them is
$$
\omega^{MW}_P(V,W):=\int_P i_V i_W\mu\,.
$$
Here $i_V i_W\mu$ is an $(n-2)$-form integrated over $P$.
Note that this symplectic structure can be thought of as the ``total" averaging 
of the symplectic structures in each normal space $N_pP$ to $P$. 
(The Marsden-Weinstein structure in higher dimensions was studied in \cite{AK, Vizman}.)

\smallskip

Now we define the Hamiltonian function on those membranes
by taking their $(n-2)$-volume: $H(P)={\rm volume}(P)=\int_P \mu_P$, where $\mu_P$ is the volume form of the  metric induced from $\R^n$ to $P$.
For instance, for a closed curve $\gamma$ in $\R^3$ this Hamiltonian is the length functional discussed in Remark \ref{lengthH}.
Note that to define the MW structure one only needs the volume form on $\R^n$, while to define the Hamiltonian one does need a metric. 

\begin{thm} 
In any dimension $n\ge 3$ the Hamiltonian vector field for the Hamiltonian $H$ and the Marsden-Weinstein  symplectic structure on codimension 2 membranes $P\subset \R^n$ is 
$$
v_{H}(p)=C_n\cdot J({\bf MC}(p))\,,
$$ 
where $C_n$ is a constant, $J$ is the operator of positive $\pi/2$ rotation in every normal space $N_pP$ to $P$, and ${\bf MC}(p)$ is the mean curvature vector to $P$ at the point $p$.
\end{thm}

This statement, as well as the proof below, is valid for any Riemannian manifold $M$.
The expression of  $v_H$ via the trace of the second fundamental form without reference to the mean curvature appeared in \cite{Vizman}, Proposition 3. For 4D this theorem was obtained in \cite{Shash}.
Here and below we use the notation $C_n$ for {\it some} constant depending on the dimension  in the case of $\R^n$, or on the geometry of $M^n$ in the general case, but not on the membrane $P$.
In the theorem above $C_n=4-2n$.

\begin{proof}
Since the MW symplectic structure is the averaging of the symplectic structures in all 2-dimensional normal planes
$N_pP$, the skew-gradient (i.e. the Hamiltonian vector)
for any functional on submanifolds $P$
is obtained from its gradient field attached at $P$ by the application of the almost complex structure $J$.
The latter is the positive rotation by $\pi/2$ in each normal plane.  (Orientations of $\R^n$
and $P$ determine the orientation of $N_pP$ and hence the positive direction of rotation
in $N_pP$ is well defined.)
Thus to prove $v_{H}(p)={\rm const}\cdot J({\bf MC}(p))$
we need to show that the gradient, i.e. the first variation, of the volume functional $H(P)$ is 
$$
\frac{\delta H}{\delta P}(p)={\rm const}\cdot  {\bf MC}(p)\,.
$$

On the other hand, the fact that the mean curvature vector field is the gradient for the volume functional is well-known, see e.g. \cite{EEL}. A quick argument follows from the observation that for a variation $P_t$ 
defined by a normal vector field $W$ attached at $P$  of dimension $l$ the volume changes at the rate
$$
\frac{d}{dt}H(P_t)=-2l\,\int_P (W, {\bf MC})\,\mu_P\,.
$$
The latter equality can be verified for a variation confined to a local chart parametrizing a neighborhood of a point $p\in P$. Let $\partial_i$ be coordinate unit vectors in this chart and $\phi$  is the chart parametrization map. Then the induced metric on $P$ in local coordinates around the point $p$ is  $g_{ij}=(\phi_*\partial_i, \phi_*\partial_j)$
and  the volume variation is $d/dt \det(g_{ij})$.
By choosing the coordinates so that $g_{ij}(p)=\delta_{ij}$ at $t=0$ one has 
$$
d/dt \det(g_{ij})={\rm tr} (d/dt\, g_{ij})={\rm tr}  (L_W g_{ij})
=2\,{\rm tr} (\nabla_W \phi_*\partial_i, \phi_*\partial_j) 
=2\,{\rm tr} (\nabla_{\phi_*\partial_i} W, \phi_*\partial_j)\,.
$$ Then using integration by parts 
 one has 
 $$
 \frac{d}{dt}H(P_t)=2\int_P {\rm tr} (\nabla_{\phi_*\partial_i} W, \phi_*\partial_j)\,\mu_P
 =-2\int_P {\rm tr} (W, \nabla_{\phi_*\partial_i} \phi_*\partial_j)\,\mu_P 
=-2l\int_P (W, {\bf MC})\,\mu_P\,,
$$
since ${\bf MC}:={\rm tr} (\nabla_{\phi_*\partial_i} \phi_*\partial_j)/l$. By applying this to $P$ of dimension $n-2$ one obtains $C_n=-2(n-2)=4-2n$.
\end{proof}

The mean curvature flow is often used  to construct minimal surfaces in $\R^3$.
For hypersurfaces  it is directed along the normal, its value is given by the mean curvature, and it minimizes the hypersurface volume in the fastest way. 

\begin{defn}
{\rm
The {\it higher vortex filament equation} on submanifolds of codimension 2 in $\R^n$ 
is given by the {\it skew-mean-curvature flow}:
\begin{equation}\label{skew-mean}
\partial_t P(p)=-J({\bf MC}(p))\,.
\end{equation}
}
\end{defn}

Note that the skew-mean-curvature flow introduced this way differs by the $\pi/2$-rotation from the mean-curvature one. Respectively, it does not stretch the submanifold while moving its points orthogonally to the mean curvatures. In particular, the volume of the submanifold $P$ is preserved under this evolution, as it should, being the  Hamiltonian function of the corresponding dynamics.

\begin{remark}
{\rm
For dimension $n=3$ the mean curvature vector is the curvature vector $k  \cdot{\bf n}$ of a curve $\gamma$:
${\bf MC}=k \cdot {\bf n}$, while the skew-mean-curvature flow becomes the binormal equation:
$\partial_t \gamma=-J(k \cdot  {\bf n})=k\cdot {\bf b}$, which for arc-length parametrization is given by the equation
$\partial_t \gamma=\gamma'\times \gamma''$.
Unlike the case $n=3$, for larger $n\ge 4$ the skew-mean-curvature flow is apparently non-integrable. 
}
\end{remark}

\begin{problem}
{\rm
a) Is there an analogue of the Hasimoto transformation for any $n$ relating the higher vortex filament equation with the higher-dimensional (and already non-integrable) nonlinear Schr\"odinger equation (NLS)?

b) Is there an analogue for any $n$ of the gas dynamics equation equivalent to the vortex filament one in 3D, see \cite{AK}.

c) Are there integrable approximations of the higher filament equation (\ref{skew-mean}).
}
\end{problem}

For $n=4$ the question a) was posed in \cite{Shash}. Recall that for $n=3$ at any time $t$ the Hasimoto transformation sends a curve $\gamma(\theta)$ with curvature $k(\theta)$ and torsion $\tau(\theta)$
to the wave function $\psi(\theta)=k(\theta)\exp\{i\int^\theta \tau(\zeta)\,d\zeta\}$ satisfying the 1-dimensional NLS: $i\partial_t\psi+\psi''+\frac 12|\psi|^2\psi=0$.

\bigskip


\section{The localized induction approximation (LIA) in higher dimensions}\label{LIAanyD}

Let $P^{n-2}\subset \R^n, \, n\ge 3$ be a closed oriented submanifold of codimension 2.
Consider the vorticity 2-form $\xi_P$
supported on this submanifold: $\xi_P=C\cdot\delta_P$. We will call $P$ a higher(-dimensional) vortex filament or membrane.
Note that the exactness (or closedness) of the 2-form $\xi_P$ implies that the membrane strength $C$ is constant, while the integrals of $\xi_P$ over 2-dimensional  surfaces with boundary not intersecting $P$ are well-defined and depend only on the homology class of the boundary in the complement to $P$.
\smallskip

We would like to find the divergence-free vector field $v$ which has a prescribed  vorticity
2-form $\xi$, i.e. $\xi=dv^\flat\in \Omega^2(\R^n)$. In dimension 3, where vorticity can be regarded as a vector field, the corresponding vector potential $v$ in $\R^3$ is reconstructed by means of the Biot-Savart formula (\ref{3DBiot}). Now we are looking for its analogue in any dimension $n\ge 3$.
The statements  we are discussing in this section
were obtained for $n=4$   in  \cite{Shash}. The extentions of proofs to the general case  are presented in Appendix.

\smallskip

The singular $\delta$-type vorticity 2-form $\xi_P$ is completely defined by the submanifold $P$.
Denote by $G(q,p)$ the Green function of the Laplace operator in $\R^n$, i.e. given a point $q\in \R^n$ one has $\Delta_p G(q,p)=\delta_q(p)$, the delta-function supported at $q$.

\begin{thm}\label{dBiot}
For any dimension $n\ge 3$ the  divergence-free vector field $v$ in $\R^n$ satisfying ${\rm curl} \,v= \xi_P$ is given by the following generalized Biot-Savart formula: for any point $q\not\in P$ one has
$$
v(q):=C_n\cdot\int_P J\left({\rm Proj}_N\nabla_p G(q,p)\right)\,\mu_P(p)\,,
$$
where ${\rm Proj}_N\nabla_p G(\cdot,p)$ is the orthogonal projection of the gradient $\nabla_p G(\cdot,p)$
of the Green function $G(\cdot,p)$ to the fiber  $N_pP$ of the normal bundle to $P$ at $p\in P$, the operator $J$ is the positive rotation around $p$ by $\pi/2$ in this 2-dimensional space $N_pP$, and
$\mu_P$ is the induced Riemannian $(n-2)$-volume form on the submanifold $P\subset \R^n$.
\end{thm}

In other words, 
$$
v(q):=C_n\cdot\int_P {\rm sgrad}_p \left(G(q,p)|_{N_pP}\right)\,\mu_P(p)\,,
$$
by using the symplectic structure in $N_pP$. Here $G(q,p)|_{N_pP}$ is the restriction of the function 
$G(q,p)$ to the normal plane $N_pP$.
These formulas use the affine structure of $\R^n$, since in the integral 
averages vectors over $P$  and attaches the total at the point $q$. In the case of an arbitrary manifold $M^n$ the Biot-Savart formula is more complicated, and we will use a round-about way to obtain the LIA for any $M^n$, see Remark \ref{LIAgeneralM} below.

\medskip

Note that as the point $q$ approaches the membrane $P$ the vector field $v(q)$ may go to infinity. Consider the following truncation of the integral above. 
For $q\in P$ and given $\epsilon>0$ take the integral over $P$ for all points $p$ satisfying 
$\|q-p\|\ge \epsilon$, i.e. at the distance at least $\epsilon$ from $q$:
$$
v_\epsilon(q):=C_n \cdot \int\limits_{p\in P,\,\,\|q-p\|\ge \epsilon} J\left({\rm Proj}_N\nabla_p G(q,p)\right)\,\mu_P(p)\,.
$$

\begin{thm}{\rm (cf. \cite{Shash} for 4D)}
\label{prop:LIA}
For any dimension $n\ge 3$ the velocity field $v$ defined in Theorem \ref{dBiot}
has the following asymptotic of the truncation $v_\epsilon$: for $q\in P\subset \R^n$ one has
$$
\lim_{\epsilon\to 0} \frac{v_\epsilon(q)}{\ln\epsilon}=C_n\cdot J\left({\bf MC}(q)\right)\,.
$$
\end{thm}

By reparametrizing the time variable $t\to -(C_n\cdot\ln\epsilon)t $ to absorb the logarithmic singularity we come to the following LIA equation for a higher filament
$P\subset \R^n$.

\begin{cor}\label{cor:LIA}
The LIA approximation for a vortex membrane (or higher filament) $P$ in $\R^n$  coincides with the skew-mean-curvature flow:
$$
\partial_t P(q)=-J\left({\bf MC}(q)\right)\,,
$$
where ${\bf MC}(q)$ is the mean curvature vector at $q\in P$.
In particular, the LIA equation is Hamiltonian with respect to the Marsden-Weinstein symplectic structure and Hamiltonian function given by the volume of the membrane $P$.
\end{cor}

Consider now the energy Hamiltonian $E(v)=\frac 12\int_M (v,v)\,\mu$ for $M=\R^n$
and fast decaying divergence-free velocity vector fields $v$.
As before, let $\xi$ be the vorticity 2-form of the field $v$, i.e.  $\xi=dv^\flat$.
If the vorticity $\xi_P$ is supported on a membrane $P\subset \R^n$ of codimension 2, the corresponding energy $E(v)=\frac 12\int_{\R^n} (v,v)\,\mu$ for the velocity $v$ defined by 
${\rm curl}\,v=\xi_P$ is divergent and requires a regularization.
Consider the {\it regularized energy} 
$$
E_\epsilon(v):=\frac 12 \int_{\R^n} (v,v_\epsilon)\,\mu\,.
$$

\begin{thm}{\rm (cf. \cite{Shash} for 4D)}\label{thm:energy}
For  any dimension $n\ge 3$ the regularized energy $E_\epsilon(v)$ for the velocity of a membrane $P\subset \R^n$ has the following asymptotics:
$$
\lim_{\epsilon\to 0}  \frac{E_\epsilon(v)}{\ln\epsilon}
=C_n \cdot \int_P\mu_P=C_n\cdot {\rm volume}\,(P)\,.
$$
\end{thm}

We refer to Appendix and \cite{Shash} for details on the proofs for $\R^n$.
As we discuss in Appendix, this regularization is also valid for any Riemannian manifold $M$.

\begin{remark}\label{LIAgeneralM}
{\rm 
When one passes from smooth to singular vorticities supported on membranes of codimension 2 the Euler dynamics 
requires regularization. Correspondingly, so does the associated energy Hamiltonian. On the other hand, the corresponding symplectic structure on smooth vorticities naturally 
descends to the MW symplectic structure on submanifolds (this is how it was defined in \cite{MW})
and does not need a regularization.

This consistency explains why the hydrodynamical Euler equation remains Hamiltonian under the
localized induction approximation. Indeed,  the LIA takes the Hamiltonian Euler equation into the Hamiltonian skew-mean-curvature equation by ``keeping only the logarithmic divergences" given by the  local terms.

For any manifold $M$ the above consistency can be taken   
as the definition of the regularized dynamics, defined 
in Theorem \ref{prop:LIA} and Corollary \ref{cor:LIA}. Namely, one can employ only the 
MW symplectic structure and  regularization of the Hamiltonian, which uses only local properties of the Green function that hold for any $M$, in order to find the vortex dynamics
in the general case.
}
\end{remark}

\bigskip


\section{Singular vorticities in codimension 1: vortex sheets}

\subsection{Vortex sheets as  exact 2-forms}

\medskip

Now we return to an arbitrary manifold $M$ (with $H^1(M)=0$), but 
consider singular vorticities supported in codimension 1. Introduce the following
  
\begin{defn}
Vortex sheets 
{\rm 
are  singular exact 2-forms, i.e.   2-currents of type
$\xi=\alpha\wedge\delta_\Gamma$, where
$\Gamma^{n-1}\subset M^n$ is a closed oriented hypersurface in $M$,
$\delta_\Gamma$ is the corresponding Dirac 1-current supported on $\Gamma$, and $\alpha$ is a closed 1-form on $\Gamma$. 
}
\end{defn}

For a singular 2-form $\xi=\alpha\wedge\delta_\Gamma$ to be exact either 

$i)$ the closed 1-form
$\alpha$ must be exact on $\Gamma$, i.e. $\alpha=df$ for a function $f$ on $\Gamma$, or 

$ii)$ 
the hypersurface $\Gamma$ must be a boundary of some domain $\partial^{-1}\Gamma\subset M$ and the closed 1-form $\alpha$ has to admit an extension to a closed 1-form $\bar\alpha$ on  $\partial^{-1}\Gamma$.  (For instance, $\Gamma$ is a torus in $\R^3$ while $\alpha=d\theta$ with $\theta$ being  one of the generating angles of the torus.)
Note that under the assumption $H^1(M)=0$ a closed hypersurface
$\Gamma$ is always a boundary. 

The above options come from the interpretation
of $\xi=\alpha\wedge\delta_\Gamma$: one can choose either $\alpha$ or $\delta_\Gamma$ to be exact, while the other form closed, for the wedge product 
to be exact.
  
  \smallskip
  
For an exact form $\alpha=df$ the vortex sheet is fibered by levels of the function $f$. If $\alpha$ is a closed 1-form, it is a function differential only locally, and the integral submanifolds of $\mathrm{ker}~\alpha$ foliate $\Gamma$. Thus the vortex sheets are fibered into filaments
  (of codimension 1 in $\Gamma$) in the former case and foliated in the latter.

\begin{example}
{\rm If $\alpha$ is supported on a single hypersurface $\gamma$ in $\Gamma$ (i.e. on a curve $\gamma\subset \Gamma$ for $n=3$),
then the vortex sheet $\xi=\alpha\wedge\delta_\Gamma=\delta_\gamma$
reduces to the vorticity of the filament $\gamma\subset\Gamma$.
}
\end{example}

\begin{remark}
{\rm 
The corresponding primitive 1-forms $\eta$ satisfying
$\xi=d\eta$ for the singular vorticity 2-form $\xi=\alpha\wedge\delta_\Gamma$ are as follows. 

$i)$ For an exact $\alpha=df$ take $\eta=f\delta_\Gamma$. 

$ii)$ For a closed 1-form $\alpha$ extendable to a closed 1-form $\bar\alpha$ on a domain ${\partial^{-1} \Gamma}$  take as  a primitive $\eta=d^{-1}\xi$ the 1-form 
  $\eta=-\chi_{\partial^{-1} \Gamma}\cdot \bar\alpha$, where ${\partial^{-1}
  \Gamma}$ is a domain bounded by the hypersurface $\Gamma$ and
  $\chi_{\partial^{-1} \Gamma}$ is its characteristic function. Indeed,
$$
d\eta=-d(\chi_{\partial^{-1} \Gamma}\cdot \bar\alpha)=-d\chi_{\partial^{-1}
\Gamma}\wedge \bar\alpha=-\delta_{\Gamma}\wedge \bar\alpha=\alpha\wedge \delta_{\Gamma}=\xi\,.
$$

Note that the 1-form $\bar\alpha$ and the domain ${\partial^{-1}\Gamma}$ are not defined uniquely,
and this ambiguity corresponds to the ambiguity
in the definition of a primitive 1-form  $\eta=d^{-1}\xi$.
}
\end{remark}

\bigskip

Vortex sheets $\xi$ understood as singular  currents  can be regarded as
elements of a completion of the dual space  $\SVect(M)^*$. 
(It is convenient to change the order in this wedge product to $\xi=\delta_\Gamma \wedge\alpha$
in order to avoid the signs depending on the dimension of $M$  in the pairing and symplectic structure below.)

\begin{defn-prop}
The pairing of vortex sheets (i.e. singular vorticity currents) $\xi=\delta_\Gamma \wedge\alpha$ with vector fields $V\in \SVect(M)$ is defined by 
(cf. the pairing (\ref{pairing}))
$$
\langle d^{-1}(\delta_\Gamma \wedge\alpha), V\rangle
=\int_M i_V  d^{-1}(\delta_\Gamma \wedge\alpha) \cdot \mu\,,
$$
where $d^{-1}(\delta_\Gamma \wedge\alpha)$ is a primitive 1-form for the vorticity $\xi=\delta_\Gamma \wedge\alpha$. The pairing is well defined, i.e. it does not depend on the choice
of $d^{-1}$.
\end{defn-prop}

\begin{proof}  Indeed,
$$
\langle d^{-1}(\delta_\Gamma \wedge\alpha), V\rangle
=\int_M   d^{-1}(\delta_\Gamma \wedge\alpha) \wedge i_V\mu
=\int_M   \delta_\Gamma \wedge\alpha \wedge d^{-1}(i_V\mu)
=\int_\Gamma \alpha \wedge d^{-1}(i_V\mu)\,.
$$
Since $H^1(M)=0$ the closed $(n-1)$-form $i_V\mu$ is exact, and its primitives  
$d^{-1}(i_V\mu)$ may differ by an exact $(n-2)$-form $\zeta$. Then the form $\alpha\wedge \zeta$
is exact on $\Gamma$ and  the corresponding pairing difference given by the integral over $\Gamma$  is zero. 
\end{proof}

For instance, for an exact $\alpha=df$ the pairing reduces to 
$\langle d^{-1}(\delta_\Gamma \wedge df), V\rangle=\mathrm{Flux}\,(fV)|_\Gamma$.

\medskip

\begin{remark}
{\rm
Suppose that $\Gamma$ is the oriented boundary between two different parts
  $M_j$ with velocity fields $v_1, v_2$
that are divergence-free and vorticity-free (i.e. locally potential flows).
The vorticity is infinite at the interface $\Gamma$ and here we describe 
how to define the 1-form $\alpha$ in the corresponding vortex sheet $\xi=\alpha\wedge\delta_\Gamma$.
  
Given a Riemannian metric on $M$, we prepare the 1-form $v^\flat_j$ on
$M_j$ corresponding
to the velocity $v_j$, respectively. Note that the forms $v^\flat_j$ must be locally exact,
 $v^\flat_j=dh_j$  since
${\rm curl}\,v_j=0$ on $M_j$. Then locally $\alpha:=(dh_1-dh_2)|_{\Gamma}=df_1-df_2$.
One can also define this 1-form $\alpha=d(f_1-f_2)$ by means of the
vector field $v_\Gamma$
inside this vortex sheet $\Gamma$ by using  the metric restricted to
$\Gamma$: locally $v_\Gamma:=(d(f_1-f_2))^\sharp={\rm Proj}|_\Gamma
(v_1-v_2)$.
The proper sign of  $v_\Gamma$ or the form $\alpha$ 
depends on the orientation of $\Gamma$: the latter defines
the orientation of the corresponding exterior normal and hence
signs of the fields $v_1$ and $v_2$ in this difference.
}
\end{remark}

\bigskip


\subsection{Definition and properties of the symplectic structure on vortex sheets}


There is a natural symplectic structure on vortex sheets coming
from  the Lie-Poisson structure on $\SVect(M)^*$. It extends
the Marsden-Weinstein symplectic structure for filaments in $\R^3$ and for membranes of codimension 2 in $M^n$. The corresponding symplectic leaves are defined by isovorticed fields, i.e. fields with diffeomorphic singular vorticities $\alpha\wedge \delta_\Gamma $. 
The corresponding symplectic structure on spaces of diffeomorphic vortex sheets is defined as follows.

\medskip

\begin{defn}
{\rm 
Given two vector fields $V, W$ attached at $\Gamma$  define the symplectic structure
on variations of vortex sheets $\xi=\delta_\Gamma \wedge\alpha$, i.e. pairs $(\Gamma, \alpha)$, by 
$$
\omega_{\delta_\Gamma\wedge\alpha}(V,W):=\int_{\Gamma}\alpha \wedge  i_Vi_W\mu\,.
$$
}
\end{defn}

\begin{thm}
The form $\omega_{\delta_\Gamma\wedge\alpha}$ coincides with the Kirillov-Kostant symplectic structure $\omega_\xi^{KK}$ on the coadjoint orbit containing the vortex sheet $\xi$ in $\SVect^*(M)$.
\end{thm}

\begin{proof}  Adapt the formula (\ref{KK}) for Kirillov-Kostant symplectic structure on the coadjoint orbit of $\xi$ to the case of a vortex sheet $\xi=\delta_\Gamma \wedge\alpha$. 
Let $V$ and $W$ be two variations of  $\xi$ given by divergence-free vector fields on $M$.
Then by using the identity $i_{[V,W]} \mu=d(i_Vi_W\mu)$ valid 
for divergence-free  fields  and specifying to the case of
  $\xi=\delta_\Gamma \wedge\alpha$ one obtains  
$$
\omega^{KK}_\xi(V,W): =\int_{M} d^{-1}\xi \wedge i_{[V,W]} \mu
=\int_{M}  d^{-1}\xi  \wedge d(i_Vi_W\mu)
=\int_{M} \xi \wedge i_Vi_W\mu
$$
$$
= \int_{M} \delta_\Gamma\wedge \alpha \wedge i_Vi_W\mu
=\int_{\Gamma} \alpha\wedge i_Vi_W\mu
= \omega_{\delta_\Gamma\wedge\alpha}(V,W)\,.
$$
\end{proof}

\begin{remark}
{\rm
 If $\alpha$ is supported on a curve $\gamma\subset \Gamma$, i.e.
$\xi=\delta_\gamma$, then
$\omega_\xi(V,W):=\int_{\gamma}i_V i_W\mu$. For a curve $\gamma\subset
\R^3$ this is exactly
the Marsden-Weinsten symplectic structure $\omega^{MW}_\gamma$ on filaments, i.e. non-parametrized curves in $\R^3$, see (\ref{MWss}).
}
\end{remark}

\medskip

The evolution of  vortex sheets $\xi=\alpha\wedge\delta_\Gamma$ is
defined by the classical
Euler equation in the vorticity form $\partial_t\xi+L_v\xi=0$, where
$\xi={\rm curl}\,v=dv^\flat$. This equation
is Hamiltonian with respect to the above symplectic structure $\omega^{KK}_\xi$. The standard
energy Hamiltonian $E(v)=\frac 12 \int_M ( v,v)\,\mu$ 
defines a non-local evolution of the vortex sheet, similarly to the case of membranes.

Let $(f,\theta)$ be coordinates on a vortex sheet $\alpha\wedge\delta_\Gamma$ in $\R^3$ where the exact 1-form $\alpha=df$ and the surface $\Gamma$ is fibered into  the filaments $\Gamma_f$ being  levels of the function $f$. The rough LIA procedure similar to the one described in Section \ref{filaments} under the cut-off assumption 
$\epsilon<|\theta-\tilde \theta|\le |f-\tilde f|^2$ leads to the binormal type equation:
$
\partial_t \Gamma=\Gamma_\theta\times \Gamma_{\theta\theta}\,,
$
which is Hamiltonian with the Hamiltonian function $H(\Gamma):=\int \text{length}(\Gamma_f)\, df$. The latter may be understood as a continuous family of binormal equations.
One may hope that other assumptions on the cut-off procedure lead to more interesting approximations.
\medskip

\begin{problem}
{\rm Describe possible analogues of  localized induction
approximations (LIAs) and the  length Hamiltonian for vortex sheets. 
}
\end{problem}

The Euler evolution of vortex sheets is described in the closed form by the Birkhoff-Rott equation, see e.g. \cite{instability}. The motion of vortex sheets is known to be  subject to instabilities of Kelvin-Helmholtz type which lead to roll-up phenomena.
It would be interesting to obtain this instability within the Hamiltonian framework for vortex sheets described above, cf. \cite{Loeschcke}.

\bigskip


\section{Appendix: Derivation of the LIA in higher dimensions}

In this Appendix we outline,  following \cite{Shash} and extending it to any dimension, the generalized 
Biot-Savart formula and regularized energy for the vector fields whose vorticity is confined to membranes, i.e. submanifolds of codimension 2.

\subsection{Generalized and localized Biot-Savart formulas} 

Let $v$ be a vector field in the Euclidean space $\R^n, \,n\ge 3$.
Assume this field to be  divergence free: ${\rm div}\,v=0$ or, equivalently, $d^* v^\flat=0$ for the 1-form $v^\flat$ on $\R^n$. 
Its vorticity is the 2-form $\xi=dv^\flat$. We are looking for a generalized Biot-Savart formula 
which would allow one to reconstruct the velocity field $v$ for a given vorticity 2-form $\xi$, and in particular, for a given singular vorticity $\xi_P=\delta_P$ supported on a compact membrane $P$.

Consider $d^*\xi=d^* d v^\flat=\Delta v^\flat=(\Delta v)^\flat$. 
Then component-wise one has the Poisson equation on $v$:
$\Delta v_i=  *(dx_i\wedge * d^*\xi)$.

Let $G(\cdot ,p)$ be the Green function for the Laplace operator in $\R^n\ni p$.
Then at any $q\in \R^n$ the components of the field-potential are
$$
v_i(q)= \int_{\R^n} G(q,p)\wedge *(dx_i\wedge * d^*\xi)(p)\,\mu(p)
=- \int_{\R^n} (\partial_i, (*(d_p G(q,p)\wedge *\xi))^\sharp)\,\mu(p)\,,
$$
where $\partial_i$ are the coordinate unit vectors in $\R^n$.  The vector  field-potential $v$ itself is
\begin{equation}\label{genBiot}
v(q)=- \int_{\R^n} (*(d_p G(q,p)\wedge *\xi))^\sharp \,\mu(p)\,,
\end{equation}
which is the {\it generalized Biot-Savart formula} in the case of smooth vorticity $\xi$.

\medskip

\begin{thm}\label{dBiot-app} {\rm ({\bf = \ref{dBiot}$'$})} 
For any dimension $n$ the  divergence-free vector field $v$ satisfying ${\rm curl} \,v= \xi_P$ (i.e.  $dv^\flat=\xi_P$) is given by the following localized Biot-Savart formula: for any point $q\not\in P$ one has
\begin{equation}\label{localBiot}
v(q):= \int_P {\rm sgrad}_p \left(G(q,p)|_{N_pP}\right)\,\mu_P(p)\,,
\end{equation}
where  $G(q,p)|_{N_pP}$ is the restriction of the function 
$G(q,p)$ to the normal plane $N_pP\subset \R^n$.
\end{thm}

\begin{proof}
In order to set $\xi$ to be $\xi_P= \delta_P$ we think of the latter in terms of local coordinates. Let $t_1,...,t_{n-2}$ be local coordinates along  $P$, while $\nu_1, \nu_2$ are coordinates normal to $P$  near $p\in P$. Then locally 
$\xi_P(q)=C\delta_p(q) \,d\nu_1\wedge d\nu_2$ where $\delta_p(q)$ is a delta-function supported at  $p\in P$, 
i.e. $\xi_P$ is the $\delta$-type 2-form in the transversal to $P$ direction, and one has
$*\xi_P=\mu_P$.

\smallskip

Then for $q\in \R^n$ and $p\in P$ one has
$$
v(q)=- \int_{\R^n} \delta_p(q) (*(d_p G(q,p)\wedge \mu_P))^\sharp \,\mu(p)
$$
$$
=-\int_{\R^n} \delta_p(q) \left(\frac{\partial_p G(q,p)}{\partial \nu_1}\,d\nu_2-\frac{\partial_p G(q,p)}{\partial \nu_2}\,d\nu_1\right)^\sharp  \,\mu(p)
$$
$$
=-\int_{P}  \left(\frac{\partial_p G(q,p)}{\partial \nu_1}\,d\nu_2-\frac{\partial_pG(q,p)}{\partial \nu_2}\,d\nu_1\right)^\sharp  \,\mu_P(p)=\int_{P}  {\rm sgrad}_p (G(q,p)|_{N_pP})\,\mu_P(p)\,,
$$
where the last equality is due to the following (all derivatives of $G(q,p)$ are in $p$, so we skip the index):
$$
\left(-\frac{\partial G}{\partial \nu_1}\,d\nu_2 + \frac{\partial G}{\partial \nu_2}\,d\nu_1\right)^\sharp 
=- \frac{\partial G}{\partial \nu_1}\partial_{\nu_2} +\frac{\partial G}{\partial \nu_2}\partial_{\nu_1}
=:{\rm sgrad}_p (G|_{N_pP})=J\left({\rm Proj}_N\nabla_p G\right)
$$
\end{proof}

\begin{remark}
{\rm 
For $q\not\in P$ the integrand expression above is smooth, since so is $G(q,p)$ as a function of $p\in P$. For $q\in P$ the integral (\ref{localBiot}) is well defined provided that the integration over $P$ is replaced by that over $P_\epsilon=\{p\in P~|~\|p-q\|\ge \epsilon\}.$ As $p\to q\in P$ the Green function has a singularity $G(q,p)= C_n \|{\bf r}\|^{2-n}$ where
${\bf r}:=p-q\in \R^n$. Hence
$\nabla_p G= C_n \,{\bf r}/\|{\bf r}\|^{n}$, and therefore the integral is divergent. (Recall that $C_n$ stands for any constant depending on $n$.) This divergence is ``local"
in the sense that the contributions from $p\in P$ close to $q\in P$ make the velocity $v(q)$ divergent, and this local contribution into $v(q)$ is exactly what the LIA takes into account.
}
\end{remark}

\bigskip

\subsection{Regularization of velocity}

Given $\epsilon>0$ consider a geodesic ball $U_\epsilon$ in the membrane $P$ of radius $\epsilon$ around a point $q\in P$. Define now a truncation 
 $v_\epsilon(q)$  by integrating in  (\ref{localBiot})  over $P_\epsilon:=P\setminus U_\epsilon$
 instead of over $P$:
$$
v_\epsilon(q):=\int_{P_\epsilon} {\rm sgrad}_p \left(G(q,p)|_{N_pP}\right)\,\mu_p\,.
$$

\begin{thm}{\rm ({\bf = \ref{prop:LIA}$'$})} 
For any dimension $n$ the approximation $v_\epsilon(q)$ 
has the following asymptotics: at any point $q\in P$ one has
$$
\lim_{\epsilon\to 0} \frac{v_\epsilon(q)}{\ln\epsilon}=C_n\cdot J\left({\bf MC}(q)\right)\,,
$$
where ${\bf MC}(q)$ is the mean curvature vector of the membrane $P$ at $q$ and 
the constant $C_n$ depends on $n$ only.
\end{thm}


\begin{proof}
In order to find the asymptotics of how $v_\epsilon(q)\to\infty$ as $\epsilon\to 0$ we localize
$v_\epsilon$, i.e. confine the integration to a punctured neighborhood $U_{\epsilon,a}=\{p\in P_\epsilon~|~\epsilon<\|p-q\|<a\}\subset P_\epsilon$, since the integral outside of it, over $P_a=P_\epsilon\setminus U_{\epsilon,a}$ is finite.

Set the origin of $\R^n$ at $q$, denote the radius vector from $q$ to $p$ by ${\bf r}=p-q\in \R^n$.
Introduce the geodesic radial coordinate $\rho$ and spherical multi-coordinate $\Theta$ in the ball $U_a$ of radius $a$ inside the membrane $P$. Note that for the volume form on $P$ of dimension $n-2$ one has $\mu_P=\rho^{n-3}\,d\rho\,d\Theta$.

\smallskip

Then for a point  $p\in  U_{\epsilon,a}$ one has 
$G(q,p)= C_n / \|{\bf r}\|^{n-2} \sim C_n / \rho^{n-2}$ for the Green function, where  $\sim$ stands for the leading term in the corresponding expansion. 
Hence, 
$\nabla_p G(q,p)\sim  C_n \,{\bf r}/\rho^{n}$. 
Denote by $\nu_1, \nu_2$ normal coordinates to the codimension 2 membrane $P$ near $q$. We have  
$$
v_\epsilon(q)
  \sim \int\limits_{U_{\epsilon,a}} J\left( {\rm Proj}_N\left(\nabla_p G(q,p)\right)\right)\,\mu_P(p)
= C_n \int\limits_{S^{n-3}}\int_\epsilon^a J\frac{({\bf r}, \partial_{\nu_1})\partial_{\nu_1}+({\bf r},\partial_{\nu_2})\partial_{\nu_2}}{\rho^n}
                \rho^{n-3}\,d\rho\,d\Theta\,.
$$
Now we fix $\Theta$ (temporarily suppressing this notation) and denote by ${\bf t}(\rho)={\partial {\bf r}}/{\partial \rho}$ the
tangent vector to the geodesic in direction $\Theta$: expand the following quantities in $\rho$ near $\rho=0$ in the punctured neighborhood  $U_{\epsilon,a}$
as follows:
$$
\partial_{\nu_i}(\rho)=\partial_{\nu_i}(0)+\rho\frac{\partial_{\nu_i}}{\partial \rho}(0)+{\mathcal O}(\rho^2)\,;
$$
$$
{\bf r}(\rho)=\rho\frac{\partial {\bf r}}{\partial \rho}(0) +\frac{\rho^2}{2}\frac{\partial^2 {\bf r}}{\partial \rho^2} (0)+{\mathcal O}(\rho^3)
=\rho\,{\bf t}(0) +\frac{\rho^2}{2}\frac{\partial {\bf t}}{\partial \rho}(0) +{\mathcal O}(\rho^3)\,.
$$
Then for a given $\Theta$ using 
$({\bf t}, \frac{\partial_{\nu_i}}{\partial \rho})=-(\frac{\partial {\bf t}}{\partial \rho}, \partial_{\nu_i})$, 
which is implied by $({\bf t}, \partial_{\nu_i})(\rho)=0$ for any $\rho$, one
 obtains the following expansion of  $({\bf r}, \partial_{\nu_1})\partial_{\nu_1}+({\bf r},\partial_{\nu_2})\partial_{\nu_2}$ in $\rho$:
$$
\rho^2\left[\left(\left({\bf t},\frac{\partial_{\nu_1}}{\partial \rho}\right) \partial_{\nu_1} 
                +\frac 12 \left(\frac{\partial {\bf t}}{\partial \rho}, \partial_{\nu_1}\right)\partial_{\nu_1} \right)
                + 
\left(\left({\bf t},\frac{\partial_{\nu_2}}{\partial \rho}\right) \partial_{\nu_2} 
                +\frac 12 \left(\frac{\partial {\bf t}}{\partial \rho}, \partial_{\nu_2}\right)\partial_{\nu_2} \right)
\right](0)
+{\mathcal O}(\rho^3)
$$
$$
=-\frac{\rho^2}{2}\left[ \left(\frac{\partial {\bf t}}{\partial \rho}, \partial_{\nu_1}\right)\partial_{\nu_1} 
+
\left(\frac{\partial {\bf t}}{\partial \rho}, \partial_{\nu_2}\right)\partial_{\nu_2} \right](0)
+{\mathcal O}(\rho^3)
=-\frac{\rho^2}{2} {\bf curv}_{\bf t}(0)+{\mathcal O}(\rho^3)\,.
$$
Here ${\bf curv}_{\bf t}(0)$  is  the vector of the geodesic curvature 
for the direction ${\bf t}$ at $\rho=0$ and fixed $\Theta$, i.e.  at the point $0\in \R^n$, which stands for $q$.
Restoring the dependence on $\Theta$ we have
$$
v_\epsilon(q)
\sim  C_n \int_{S^{n-3}}\int_\epsilon^a 
J\,\frac{\rho^2\,{\bf curv}_{\bf t}(0, \Theta)}{\rho^{n}}{\rho^{n-3}}\,d\rho\,d\Theta
$$
$$
=  C_n \int_\epsilon^a \frac{d\rho}{\rho} \cdot J \int_{S^{n-3}} {\bf curv}_{\bf t}(0, \Theta)\,d\Theta
\sim C_n \cdot \ln\epsilon\cdot J({\bf MC}(q) )
$$
by Definition \ref{defMC}b) of the mean curvature vector.
\end{proof}

\bigskip

\subsection{Regularization of energy}

Obtain now a regularized expression for the corresponding energy of the velocity field $v$.
Recall that the kinetic energy of a fluid moving with velocity $v$ in a manifold $M$ with a Riemannian volume form $\mu$  is $E(v)=\frac 12\int_M(v,v)\mu= \frac 12 \int_M v^\flat\wedge *v^\flat $.

Let  $\xi_P$ be the vorticity 2-form supported on a membrane $P\subset \R^n$ of codimension 2.
As we will see below, the corresponding energy $E(v)=\frac 12\int_M (v,v)\,\mu$ for the velocity $v$ satisfying ${\rm curl}\,v=\xi_P$
is divergent. Following \cite{Shash} define  the regularized energy 
$$
E_\epsilon(v)=\frac 12 \int_{\R^n} (v ,v_\epsilon)\,\mu\,.
$$

\begin{thm}{\rm ({\bf = \ref{thm:energy}$'$})} 
For  any dimension $n$ the regularized energy $E_\epsilon(v)$ has the following asymptotics:
$$
\lim_{\epsilon\to 0}  \frac{E_\epsilon(v)}{\ln\epsilon}=C_n \cdot \int_P\mu_P
=C_n\cdot {\rm volume}\,(P)\,.
$$
\end{thm}

\begin{proof}
First for any vector field $v$ we rewrite the energy $E(v)$ via vorticity by introducing the form-potential:
$$
E(v)=\frac 12 \int_M v^\flat\wedge *d^* \beta =\frac 12 \int_M \xi\wedge *  \beta 
$$
for the closed 2-form $\beta$ satisfying $d^* \beta=v^\flat$ or, equivalently, 
$\Delta \beta=dd^* \beta= dv^\flat=\xi$. Given the vorticity 2-form $\xi$, the Poisson equation 
$\Delta \beta=\xi$ on the 2-forms is equivalent to  Poisson equations for their respective components:
$\Delta \beta_{ij}=\xi_{ij}$. 
Then $\beta$ can be reconstructed component-wise by using the Green function:
$\beta_{ij}(q)=\int_{\R^n} G(q,p)\,\xi_{ij}(p)\,\mu(p)$.

\medskip

For $\xi=\xi_P$ and  $\Delta \beta=\xi_P=C\cdot \delta_P$, one has 
$\beta_{\nu_1 \nu_2}(q)=C\int_P G(q,p)\mu_P(p)$ for the normal to $P$ 
component of the potential $\beta$, while other components are zero. 
Here $\mu_P$ is the volume form induced from $\R^n$ to $P$.

By plugging this to the formula 
$E(v) =\frac 12 \int_M \xi\wedge *  \beta $ and using $*\xi_P=C\cdot \mu_P$ we obtain 
$$
E(v)=\frac C2 \int_P  \beta_{\nu_1 \nu_2}\, \mu_P 
=\frac{C^2}{2} \int_{q\in P}  \int_{p\in P}  G(q,p)\, \mu_P(p)\, \mu_P(q)\,.
$$ 
The latter is a divergent integral, which can be regularized by considering \break
$E_\epsilon(v)=\frac 12 \int_{\R^n} (v ,v_\epsilon)\,\mu$.
Namely, given a point $q\in P$ replace the inner integral over $P$ 
by the one over $P_\epsilon=\{p\in P~|~\|q-p\|\ge\epsilon\}$
by  removing from $P$  the $\epsilon$-neighborhood of $q$.
Then one has 
$$
E_\epsilon(v)=\frac{C^2}{2} \int_{q\in P}  \int_{p\in P_\epsilon}  G(q,p)\, \mu_P(p)\, \mu_P(q)\,.
$$
As $\epsilon\to 0$ the inner integral $\int_{p\in P_\epsilon}  G(q,p)\, \mu_P(p)$ increases as
$C_n\cdot \ln \epsilon$, where the constant $C_n$ depends on dimension $n$ only. Indeed, 
as $p\to q$ one has $G(q,p)= C_n\,\|q-p\|^{2-n}\sim C_n\, \rho^{2-n}$, where $\rho$ is the geodesic distance from $p$ to $q$ in the membrane $P$.
Then the  integration of $G(q,p)$ in the spherical coordinates  over a small  $(n-2)$-dimensional punctured neighborhood $U_{\epsilon, a}$ of radius $a$  
around the point $q\in P$ in the membrane $P$ gives the integral 
$$
\int_{p\in U_{\epsilon, a}}  G(q,p)\, \mu_P(p) \sim
C_n\int_\epsilon^a \rho^{2-n} \rho^{n-3} \,d\rho= C_n\int_\epsilon^a \rho^{-1} \,d\rho= -C_n \ln\epsilon +\mathcal O(1)
$$ 
as $\epsilon\to 0$.
Then after the second integration over $q\in P$ the regularized energy $E_\epsilon(v)$
has the following asymptotics: 
$$
E_\epsilon(v)=C_n\cdot \int_{q\in P} (\ln \epsilon) \, \mu_P(q)+\mathcal O(1)=C_n\cdot \ln \epsilon\cdot{\rm volume}(P) +\mathcal O(1)\qquad{\rm as}\quad \epsilon\to 0\,,
$$
which completes the proof.
\end{proof}

\begin{remark}
{\rm
One can see from the proof 
that the logarithmic singularity of the energy $E_\epsilon(v)$ comes from close points in $P$.
To find the asymptotics one specifies a small parameter $a$ giving the ``range of interaction"
and send $\epsilon\to 0$, while other pairs of points do not contribute to the leading term in the expansion of  $E_\epsilon(v)$. This explains the term  ``localized induction approximation" (LIA).

Renormalize time and regard $H(P):={\rm volume}(P)$ as the new energy associated with fluid motions whose vorticity is supported on the membrane $P$. As we discussed above, this leads to the Hamiltonian dynamics of the membrane given by the skew-mean-curvature flow in any dimension.

Note that in the regularization above one essentially uses only the symmetry and the order of  
singularity of the Green function $G(q,p)$ as $\|q-p\|\to 0$. The same asymptotics of the Green function holds for an arbitrary manifold $M^n$ and so does the energy regularization which results in  $H(P)={\rm volume}(P)$.
}
\end{remark}

\bigskip


\begin{acknowledgements}
I am grateful to A.B.~Givental for a beautiful translation of the piece from ``Popov's Dream" by A.K.~Tolstoy. I thank G.~Misiolek and C.~Loeschcke for fruitful discussions and helpful remarks, the anonymous referee for various suggestions to improve the exposition, and  the MPIM in Bonn and
the IAS in Princeton for their kind hospitality. This work was partially supported by The Simonyi Fund and an NSERC research grant.
\end{acknowledgements}



\bigskip
\end{document}